\documentclass[12pt,reqno]{amsart}

\usepackage{color}
\usepackage{esint,amssymb}
\usepackage{graphicx}
\usepackage{tikz}

\usepackage{hyperref}

\newtheorem{theorem}{Theorem}
\newtheorem{proposition}[theorem]{Proposition}
\newtheorem{lemma}[theorem]{Lemma}

\theoremstyle{definition}
\newtheorem{remark}[theorem]{Remark}

\numberwithin{equation}{section}
\numberwithin{theorem}{section}

\newcommand{\R}{\mathbb{R}}

\newcommand{\HH}{\mathcal{H}}

\newcommand{\ep}{\varepsilon}

\newcommand{\J}{\mathcal A}
\newcommand{\la}{\lambda}

\newcommand{\abs}[1]{\left| #1 \right|}

\newcommand{\average}{{\mathchoice {\kern1ex\vcenter{\hrule height.4pt
width 6pt depth0pt} \kern-9.7pt} {\kern1ex\vcenter{\hrule
height.4pt width 4.3pt depth0pt} \kern-8pt} {} {} }}
\newcommand{\ave}{\average\int}

\renewcommand{\tilde}{\widetilde}

\begin{document}

\title[Estimates involving the radial derivative]
{Estimates controlling a function by only its radial derivative and applications to stable solutions of elliptic equations}

\author[X. Cabr\'e]{Xavier Cabr\'e}
\address{﻿X.C.\textsuperscript{1,2,3} ---
\textsuperscript{1}ICREA, Pg.\ Lluis Companys 23, 08010 Barcelona, Spain \& 
\textsuperscript{2}Universitat Polit\`ecnica de Catalunya, Departament de Matem\`{a}tiques and IMTech, 
Diagonal 647, 08028 Barcelona, Spain \&
\textsuperscript{3}Centre de Recerca Matem\`atica, Edifici C, Campus Bellaterra, 08193 Bellaterra, Spain
}
\email{xavier.cabre@upc.edu}

\begin{abstract}
We establish two new estimates which control a function in~$L^1$, after subtracting its average,  by only the $L^1$ norm of its radial derivative. While the interior estimate holds for all superharmonic functions, the boundary version is much more delicate. It requires the function to be a stable solution of a semilinear elliptic equation with a nonnegative, nondecreasing, and convex nonlinearity. As an application, our estimates provide quantitative proofs of two results established by contradiction-compactness arguments in [Cabr\'e, Figalli, Ros-Oton, and Serra, Acta Math.\ 224 (2020)]. We recall that this work proved the H\"older regularity of stable solutions to semilinear elliptic equations in the optimal range of dimensions $n\leq 9$. 
\end{abstract}

%\keywords{}
%\subjclass[2010]{35B65, 35B35}
%\date{\today}

\thanks{The author was supported by grants PID2021-123903NB-I00 and RED2022-134784-T funded by MCIN/AEI/10.13039/501100011033 and by ERDF ``A way of making Europe'', and by the Catalan grant 2021-SGR-00087, as well as by the Spanish State Research Agency through the Severo Ochoa and Mar\'{\i}a de Maeztu Program for Centers and Units of Excellence in R\&D (CEX2020-001084-M)}

\maketitle

\section{Introduction \vspace{.15cm}}

In 2020, the paper \cite{CFRS} by Figalli, Ros-Oton, Serra, and the author established the H\"older regularity of stable solutions to semilinear elliptic equations in the optimal range of dimensions $n\leq 9$. It concerned the semilinear elliptic equation
\begin{equation}\label{eq:PDE}
-\Delta u=f(u) \qquad \text{in }\Omega\subset \R^n,
\end{equation}
where $\Omega$ is a bounded domain, $f\in C^1(\R)$, and  $u:\overline\Omega \subset \R^n\to \R$.
Note that \eqref{eq:PDE} is the first variation of the functional
\begin{equation}\label{energy}
\mathcal E(u):=\int_{\Omega}\Bigl(\frac{|\nabla u|^2}2-F(u)\Bigr)\,dx,
\end{equation}
where $F'=f$. 
Consider the second variation of $\mathcal E$ at $u$, given by
$$
\frac{d^2}{d\varepsilon^2}\Big|_{\varepsilon=0}\mathcal E(u+\varepsilon\xi)
=\int_{\Omega}\Bigl(|\nabla \xi |^2-f'(u)\xi^2\Bigr)\,dx.
$$
One says that $u$ is a {\it stable} solution of equation \eqref{eq:PDE} in $\Omega$ if the second variation at~$u$ is nonnegative, namely, if 
\begin{equation}\label{stabilityLip} 
\int_{\Omega} f'(u) \xi^2\,dx\leq \int_{\Omega} |\nabla \xi|^2\,dx  \quad \mbox{ for all }
\xi\in C^1(\overline\Omega) \text{ with } \xi_{|_{\partial\Omega}}\equiv 0.
\end{equation}
For the purpouses of this paper, since we are concerned with estimates, we may assume
that $u \in C^\infty(\overline\Omega)$ ---in which case all expressions above are well defined.

Note that the stability of $u$ is defined within the class of functions $u+\varepsilon\xi$ agreeing with~$u$ on the boundary $\partial\Omega$. Stability is equivalent to assuming the nonnegativeness of the first
Dirichlet eigenvalue in~$\Omega$ for the linearized operator of \eqref{eq:PDE} at~$u$, namely $-\Delta-f'(u)$.

If a {\it local minimizer} of $\mathcal E$ exists (that is, a minimizer for small perturbations having same boundary values), then it will be a stable solution as well.

See the monograph \cite{Dup} by Dupaigne for more information and results on stable solutions.

The following is the result of \cite{CFRS} on interior regularity. Here, and throughout the paper, by {\it ``dimensional constant''} we mean a constant that depends only on $n$. 

\begin{theorem}\cite[Theorem 1.2]{CFRS}.
\label{thm:0}
Let $u\in C^\infty(\overline B_1)$ be a stable solution of $-\Delta u=f(u)$ in $B_1\subset \R^n$, for some nonnegative function $f\in C^{1}(\R)$.

Then,
\begin{equation}
\label{eq:W12 L1 int}
\|\nabla u\|_{L^{2+\gamma}({B}_{1/2})} \le C\|u\|_{L^1(B_1)}
\end{equation}
for some dimensional constants $\gamma>0$ and $C$.
In addition,
\begin{equation}
\label{c}
 \|u\| _{C^\alpha(\overline{B}_{1/2})}\leq C\|u\|_{L^1(B_1)} \quad \text{ if } n \leq 9,
 \end{equation}
where $\alpha>0$ and $C$ are dimensional constants.
\end{theorem}
  
The proof of the H\"older estimate \eqref{c} starts from the key bound
\begin{equation}\label{eq:11}
\int_{B_\rho}r^{2-n} u_r^2\,dx \leq C\int_{B_{3\rho/2}\setminus B_\rho} r^{2-n} |\nabla u|^2\,dx,
\end{equation}
for all $\rho< 2/3$, which holds for $3 \leq n\leq 9$ and all stable solutions. It is easily proved in Lemma~2.1 of \cite{CFRS} using the stability condition \eqref{stabilityLip} with the choice of test function $\xi=x\cdot \nabla u\,\, r^{(2-n)/2}\zeta= u_r\, r^{1+(2-n)/2}\zeta$, where $\zeta$ is a standard cut-off function in $B_1$. Here and
throughout the paper we will use the notation
\begin{equation}\label{notation r}
r=|x| \qquad\text{and}\qquad u_r(x)=\frac{x}{|x|}\cdot \nabla u(x).
\end{equation}

To conclude H\"older regularity from \eqref{eq:11}, by the hole-filling technique (see \cite{CFRS} or the expository paper \cite{C}) it suffices to replace the gradient $|\nabla u|^2$ integrated on the annulus in its right-hand side by the radial derivative $u_r^2$, also in an annulus.  Note here that the quantities in \eqref{eq:11} are rescale invariant, as well as the class of all stable solutions for some nonnegative nonlinearity. Hence, to prove that
\begin{equation*}
 \int_{B_{3\rho/2}\setminus B_\rho} r^{2-n} |\nabla u|^2\,dx
\leq C \int_{B_{3\rho/2}\setminus B_\rho}r^{2-n} u_r^2\,dx,
\end{equation*}
for every $\rho <2/3$ and stable solution $u$ (under, perhaps, certain additional assumptions on $u$), it is enough to show it for $\rho=1$ ---in which case the weight $r^{2-n}$ is both bounded from above and below in the annulus $B_{3/2}\setminus B_1$. This was accomplished in  the following result of~\cite{CFRS}.

\begin{lemma}\cite[Lemma 3.1]{CFRS}.\label{int-doubling}
Let $u\in C^2(B_2)$ be a stable solution of $-\Delta u=f(u)$ in $B_2\subset \R^n$, with $f$ locally Lipschitz and nonnegative. 
Assume that
\begin{equation}\label{doublingCFRS}
\int_{B_1}|\nabla u|^2\,dx\geq \delta \int_{B_2}|\nabla u|^2\,dx
\end{equation}
for some $\delta>0$.

Then, there exists a constant $C_\delta$, depending only on $n$ and $\delta$, such that
\begin{equation}\label{step3CFRS}
\int_{B_{3/2}}|\nabla u|^2 \,dx\leq C_\delta \int_{B_{3/2}\setminus B_1} u_r^2\,dx.
\end{equation}
\end{lemma}

\begin{remark}\label{rem:base-point}

Before completing \cite{CFRS}, we (as well as some expert colleagues) wrongly thought that one could try to replace $u_r^2$ by $|\nabla u|^2$ in the left-hand side of \eqref{eq:11} as follows ---avoiding then the need of Lemma~\ref{int-doubling}. Replace the origin on the left-hand side of \eqref{eq:11} by any other point $y$, obtaining then the same inequality with the new left-hand side $\int_{B_\rho(y)}r_y^{2-n} u_{r_y}^2\,dx$ for all points $y$ in a certain ball around the origin, where $r_y(x)=|x-y|$ and $u_{r_y}$ is the radial derivative centered at $y$. Averaging in $y$ one could think that $r^{2-n}u_r^2$ on the left-hand side of \eqref{eq:11} might then be replaced by $r^{2-n}|\nabla u|^2$  ---this would be true if $n=2$ and hence the weight $r^{2-n}$ were not present. However, as shown in  Appendix E of~\cite{C} with some concrete examples, such averaging procedure cannot work for $n\geq 3$, due to the presence of the weight~$r_y^{2-n}$.
\end{remark}

The proof of the lemma was based on a contradiction-compactness argument. Compactness in the $W^{1,2}$ norm followed from the higher integrability bound \eqref{eq:W12 L1 int}, stated in $B_{3/2}$ instead of $B_{1/2}$. The contradiction came from analyzing the validity of \eqref{step3CFRS} in the limiting case $u_r\equiv 0$. This is simple, since if $u_r\equiv 0$, then $u$ is $0$-homogeneous and, thus, $u$ is a superharmonic function for the Laplace-Beltrami operator on the unit sphere ---recall that $f\geq 0$. This yields $u$ to be constant, and hence $|\nabla u|\equiv 0$.

In the current paper we succeed to avoid such a contradiction-compactness proof combining estimate \eqref{eq:11} with two more ingredients. The first one is \cite[Proposition 2.5]{CFRS}, which claims that stable solutions, when $f\geq 0$, satisfy the bound
$$
\|\nabla u\|_{L^{2}(B_{1/2})}  \le C \|u\|_{L^{1}(B_{1})}.
$$
This estimate was proved in \cite{CFRS} in a quantitative way by using, among other ideas, the test function $\xi=|\nabla u|\zeta$, with $\zeta$ a simple cut-off function, in the stability inequality \eqref{stabilityLip}. With this estimate at hand, we cover the annulus on the right-hand side of \eqref{eq:11} by sufficiently small balls and use the estimate (applied to $u$ minus its average in the annulus, which is still a stable solution of a semilinear equation for a new nonnegative nonlinearity) in each ball. Adding all inequalities, we conclude that it suffices to bound the $L^1$ norm of the solution (minus its average) in an annulus ---instead of the $L^2$~norm of its gradient as in \eqref{eq:11}---  by only its radial derivative.

This is the content of the first result of the current paper. It is a new inequality (we could call it {\it ``a radial Poincar\'e inequality''}) that will be established in a quantitative way and that will suffice to complete the proof of the H\"older estimate \eqref{c} in Theorem~\ref{thm:0}, thus avoiding the use of  Lemma~\ref{int-doubling} and its contradiction-compactness proof. We denote by $\ave$ the average of functions in a set.

\begin{theorem}\label{thm:3}
Let $u\in C^\infty (\overline B_1)$ be superharmonic in $B_1\subset \R^n$. 

Then, 
\begin{equation}\label{new-int-ur} 
\Vert u-\textstyle{\ave_{B_1\setminus B_{1/2}}}u \Vert_{L^{1}({B_{1}\setminus B_{1/2})}}\leq C\Vert u_{r}\Vert_{L^{1}({B_{1}\setminus B_{1/2})}}
\end{equation}
and
\begin{equation}\label{new-int-ur-ball} 
\Vert u-\textstyle{\ave_{B_1}}u\Vert_{L^{1}({B_{1})}}\leq C\Vert u_{r}\Vert_{L^{1}({B_{1})}}
\end{equation}
for some dimensional constant $C$.
\end{theorem}

The estimates, which we could not find in previous literature, hold for all superharmonic functions.  The fact that superharmonicity suffices is in contrast with the much more delicate case of boundary regularity, treated below, where the analogue estimate will crucially require $u$ to be a stable solution of a semilinear equation for a certain class of nonlinearities.

Note also that, with respect to \eqref{step3CFRS}, we have replaced $L^2$ norms by $L^1$ norms ---indeed, with $L^2$ norms the result would hold only for $n\leq 4$; see~Remark~\ref{Lp-version} below.
In addition, the new estimates do not require a doubling assumption, unlike \eqref{step3CFRS} ---which required \eqref{doublingCFRS}. This will simplify the H\"older regularity proof of \cite{CFRS}, as presented in all detail in the expository paper \cite{C}. More importantly,  we will establish Theorem~\ref{thm:3} with a quantitative proof. In particular, this will allow to quantify the H\"older exponent~$\alpha$ in Theorem~\ref{thm:0}. The proof will use only estimate \eqref{new-int-ur} in an annulus ---estimate \eqref{new-int-ur-ball} in balls will not be needed; it is established in this paper as additional information and possible future use.

In fact, regarding estimates \eqref{new-int-ur} and  \eqref{new-int-ur-ball}, one could wonder about their validity for all functions~$u$, and not only superharmonic ones. For $n\geq2$, both inequalities would be false. Indeed, this can be seen with the family of smooth functions 
$$
u^\ep(x):=\varphi(|x|/\ep)\, x_n/|x|, 
$$
where $\ep\in (0,1)$, $\varphi$ is smooth and nonnegative, vanishes in $(0,1/4)$, and is identically 1 in $(1/2,\infty)$. Then, $u^\ep_r$ has support in $B_{\ep/2}\setminus B_{\ep/4}$ and its $L^1(B_1)$ norm is bounded above by $C\ep^{n-1}$. Instead, to control the $L^1(B_1\setminus B_{1/2})$ norm of $u^\ep-\textstyle{\ave_{B_1\setminus B_{1/2}}} u$ (or, similarly, of $u^\ep-\textstyle{\ave_{B_1}} u$) by below, if $\textstyle{\ave_{B_1\setminus B_{1/2}}} u\leq0$ (the case $\textstyle{\ave_{B_1\setminus B_{1/2}}} u\geq 0$ is done similarly) we may integrate only in $\{1/2<|x|<1, x_n>0\}$. In this set $|u^\ep-\textstyle{\ave_{B_1\setminus B_{1/2}}} u|=u^\ep -\textstyle{\ave_{B_1\setminus B_{1/2}}} u\geq u^\ep=x_n/|x|$, and thus the $L^1(B_1\setminus B_{1/2})$ norm is larger than a positive constant independent of $\ep$.

\begin{remark}\label{vanish-on-bdry}
Instead, it is easy to check (this will be done in \cite{CS}) that the radial Poincar\'e inequalities \eqref{new-int-ur} and \eqref{new-int-ur-ball} hold among all functions that vanish on $\partial B_1$ (in this case, there is no need to subtract a constant to the function $u$ for the inequalities to hold). They also hold for general harmonic functions, not necessarily vanishing on  $\partial B_1$.
\end{remark}

\begin{remark}\label{Lp-version}
In the forthcoming paper \cite{CS}, Saari and the author determine the range of exponents $p\in [1,\infty)$ for which the radial Poincar\'e inequalities \eqref{new-int-ur} and \eqref{new-int-ur-ball} hold, for superharmonic functions, with $L^1$ replaced by $L^p$ in both their left and right-hand sides. It turns out that inequality \eqref{new-int-ur} in an annulus holds if and only if
$$
\text{either } n\leq 3, \text{ or } n\geq 4 \text{ and } 1\leq p <(n-1)/(n-3).
$$
Instead, \eqref{new-int-ur-ball} in a ball holds if and only if
$$
\text{either } n\leq 3, \text{ or } n\geq 4 \text{ and either } 1\leq p <(n-1)/(n-3) \text{ or } p>n.
$$

On the other hand, the statements of Remark~\ref{vanish-on-bdry} not only hold for $p=1$, but also for every $p\in [1,\infty)$ (see \cite{CS}).
\end{remark}

Let us also point out that having two different proofs of the regularity result of \cite{CFRS} has been already of interest in the context of other equations. Indeed, the compactness proof of \cite{CFRS} has been extended to the $p$-Laplacian in \cite{CMS}. Instead, Theorem \ref{thm:3} (and Theorem~\ref{thm:0}) have been extended to stable solutions for operators with variable coefficients by Erneta~\cite{Ern1}. Our quantitative proof of H\"older regularity up to the boundary (presented below) has also been extended by Erneta~\cite{Ern2, Ern3} to operators with variable coefficients. As mentioned in his papers, this has allowed to avoid the intricate blow-up and Liouville theorem result form \cite{CFRS}, as well as a delicate compactness result (Theorem~4.1 in \cite{CFRS}) needed to treat boundary regularity.

 \begin{remark}\label{fbddbelow} 
Our quantitative proof of Theorem~\ref{thm:0} allows for an easy extension of the result to the case $f\geq -K$ for some positive constant $K$ ---note that Theorem~\ref{thm:0} assumed $f\geq 0$. The resulting estimates in this case are those of Theorem~\ref{thm:0} with the right-hand sides replaced by $C(\|u\|_{L^1(B_1)}+K)$, where $C$ is still a dimensional constant. The necessary changes in the proofs are explained in Section~7 of~\cite{C}.

It was not obvious how to obtain the result for $f\geq -K$ by extending the proof in \cite{CFRS} ---due to the compactness argument used in that paper. However, this has been accomplished recently by Fa Peng \cite{P}, who has weaken the hypothesis  $f\geq 0$ in \cite{CFRS} to allow $f\geq A\min(0,t)-K$ for positive constants $A$ and~$K$. This improves our requirement $f\geq -K$. He needs, however, to replace the~$L^1$ norm of $u$ on the right-hand side of \eqref{c} by the $W^{1,2}$ norm of $u$.
\end{remark}

The proof of Theorem \ref{thm:3} uses a harmonic replacement together with the maximum principle. This allows us to reduce the problem to the case when $u$ is harmonic ---for this reduction it will be crucial that we work with $L^1$ instead of $L^2$ norms. But when $\Delta u=0$, we confront a Neumann problem for the homogeneous Laplace equation (since $u_r$, restricted to any sphere, may be thought as the Neumann data), for which an estimate on spheres similar to the one of the theorem is not surprising and, in fact, easy to prove.

Let us now turn into the much more delicate case of boundary regularity.
The main result in \cite{CFRS} established a $C^\alpha$ estimate up to the boundary for stable solutions of the problem
\begin{equation}\label{Dirpb}
\left\{
\begin{array}{cl}
-\Delta u=f(u) & \text{in }\Omega\\
u=0 & \text{on }\partial\Omega,
\end{array}
\right.
\end{equation}
in bounded domains $\Omega$ of class $C^3$, dimensions $n\leq 9$, and nonlinearities $f$ that are nonnegative, nondecreasing, and convex. Erneta~\cite{Ern2,Ern3} has improved this result, thanks to the use of the quantitative techniques of the current article, to allow for $C^{1,1}$ domains. 

In \cite{CFRS}, the key boundary result was Theorem 6.1 in that paper. To state it for the case of half-balls, let us introduce the notation
$$
\R^n_+=\{x\in\R^n\, :\, x_n>0\}, \quad B_\rho^+=\R^n_+\cap B_\rho, \quad\text{and}\quad  \partial^0 B^+_\rho = \{x_n=0\}\cap \partial B^+_\rho.
$$

\begin{theorem}\cite[Theorem 6.1 with $\theta=0$]{CFRS}.
\label{thm:0bdry}  
Let $u\in C^\infty(\overline{B^+_1})$ be a nonnegative stable solution of 
$-\Delta u=f(u)$ in $B^+_1\subset\R^n$, with $u=0$ on $\partial^0 B^+_1$.
Assume that $f\in C^1(\R)$ is nonnegative, nondecreasing, and convex.
 
Then,
\begin{equation}
\label{holder-bdry}
\|u\| _{C^\alpha (\overline{B^+_{1/2}})}\leq C\|u\|_{L^1(B^+_1)}  \quad \text{ if } n \leq 9,
\end{equation}
where $\alpha>0$ and $C$ are dimensional constants.
\end{theorem}

It is important here to assume the solution $u$ to be nonnegative ---an assumption that will hold for problem~\eqref{Dirpb}, since $f\geq0$. Note also that the result requires more assumptions on the nonlinearity than the interior Theorem~\ref{thm:0}.

As in the interior case, the proof of Theorem \ref{thm:0bdry} given in \cite{CFRS} requires a boundary analogue of Lemma~\ref{int-doubling} on control of $\nabla u$ by only its radial derivative~$u_r$. It is here where \cite{CFRS} used a contradiction-compactness argument (Step~2 of the proof of Proposition 6.3 in that paper) in addition to a new delicate result. To explain the new result, it is interesting to first look at the extreme case in which $u_r\equiv 0$ ---this can be seen as a first validity test for a boundary version of estimate \eqref{step3CFRS}. In fact, such case had to be analyzed also in the contradiction-compactness proof of \cite{CFRS}. In contrast with the interior case, ``bad news'' arise here. Indeed, there exist nonzero nonnegative superharmonic functions $u$ in $B_1^+$ which vanish on $\partial^0  B_{1}^+$ and are 
\linebreak 
0-homogeneous (i.e., for which $u_r\equiv 0$) ---obviously, they are not smooth at the origin, but belong to $W^{1,2}(B_{1}^+)$ when $n\ge 3$. An example is given by the function 
$$
u(x)= \frac{x_n}{|x|}.
$$ 
As explained in all detail in Remark~\ref{rk:not-superh} below, 
this is the key reason behind the fact that our new boundary result, Theorem~\ref{thm:3bdry} below, does not hold in the class of all superharmonic functions which are smooth up to the flat boundary and vanish on it.

As a consequence, the contradiction in \cite{CFRS} could not come from the limiting function being superharmonic ---as in the interior case---, but from a stronger property. Namely, Theorem~4.1 of \cite{CFRS} established that the class of $W^{1,2}$ functions~$u$ for which there exists a nonnegative, nondecreasing, and convex nonlinearity~$f$ (with $f$ blowing-up, perhaps, at the value $\sup u$)
 such that $u$ is a stable solution of $-\Delta u=f(u)$, is closed under the $L^1$-convergence of the functions~$u$  in compact sets. Note that the class does not correspond to stable solutions of only one equation, but of all equations with such nonlinearities. This is a very strong  compactness result, of interest by itself. Thanks to Theorem~\ref{thm:3bdry} of the current paper, it will not be needed in the new quantitative proof of boundary H\"older regularity given in~\cite{C}. Instead, \cite{CFRS} used this compactness result to find that, in the extreme case discussed above, the limiting function $u$ would be not only superharmonic, but also a solution of an equation  $-\Delta u=g(u)$ for some nonlinearity $g$ with the properties listed above. This gave a contradiction, since then $g(u)$ would be $0$-homogeneous (being $u_r$ identically 0), while $-\Delta u$ would be $-2$-homogeneous.
 
The compactness Theorem~4.1 of \cite{CFRS} will not be needed in our quantitative proof. As a consequence, it will be simpler to implement our proof  in other frameworks than that of \cite{CFRS}. This is the case of stable solutions for operators with variable coefficients, as treated in~\cite{Ern2, Ern3}.

The new quantitative proof of boundary H\"older regularity (as presented in~\cite{C}) uses the main and most delicate result of the current paper ---more precisely, the following estimate  \eqref{introbdryradial} in a half-annulus.

\begin{theorem}\label{thm:3bdry}
Let $u\in C^\infty(\overline{B_{1}^+})$ be a nonnegative stable solution of $-\Delta u=f(u)$ in $B_{1}^+\subset\R^n$, with $u=0$ on $\partial^0  B_{1}^+$. Assume that $f\in C^1(\R)$ is nonnegative, nondecreasing, and convex.
 
Then,
\begin{equation}\label{introbdryradial}
\| u\|_{L^1(B_{1}^+\setminus B_{1/2}^+)} \le C \|u_r\|_{L^1({B_{1}^+\setminus B_{1/2}^+)}}
\end{equation}
and
\begin{equation}\label{introbdryradial-ball}
\| u\|_{L^1(B_{1}^+)} \le C \|u_r\|_{L^1({B_{1}^+)}}
\end{equation}
for some dimensional constant $C$.
\end{theorem}

Several comments are in order.

First, in Remark~\ref{rk:not-superh} we will show that, unlike the interior case of Theorem~\ref{thm:3}, estimate \eqref{introbdryradial} does not hold in the larger class of nonnegative superharmonic functions in~$B_1^+$ which are smooth up to the boundary of $B_1^+$ and vanish on $\partial^0 B^+_1$. 

In Remark~\ref{rk:bdry-ngeq3} we will prove that, at least for $n\geq3$, Theorem \ref{thm:3bdry} also holds when replacing \eqref{introbdryradial}, or  \eqref{introbdryradial-ball}, by the stronger bound
$$
\| u\|_{L^1(B_{1}^+)} \le C \|u_r\|_{L^1({B_{1}^+\setminus B_{1/2}^+)}} \quad\text{ if $n\ge 3$}.
$$

The proof of Theorem~\ref{thm:3bdry} is rather delicate and required new ideas. It is completely different than that of the analogue interior estimate, Theorem~\ref{thm:3}. It uses not only the semilinear equation for $u$, but also the stability condition.  Furthermore, it requires the nonlinearity to satisfy $f\ge 0$, $f'\ge 0$, and $f''\ge 0$. Note that these are exactly the same assumptions needed in the contradiction-compactness proof in \cite{CFRS}. This fact came partly as a surprise, since our quantitative proof and the one of  \cite{CFRS} are based in very different ingredients, as described next.

The proof of Theorem \ref{thm:3bdry} will be quantitative. It starts from the identity
\begin{equation}\label{origin:bdry}
-2 \Delta u + \Delta (x\cdot\nabla u)  = -f'(u) \, x\cdot\nabla u,
\end{equation}
which is easily checked by computing $\Delta (x\cdot\nabla u)$.
After multiplying \eqref{origin:bdry} by an appropriate test function and integrating by parts in a half-annulus, it is easily seen that the left-hand side bounds the $L^1$ norm of $u$ by below, modulus an error which is admissible for  \eqref{introbdryradial}: the $L^1$ norm of $x\cdot \nabla u=ru_r$ in a half-annulus. The main difficulty is how to control the right-hand side of \eqref{origin:bdry} by above, in terms of only~$u_r$. This is delicate, due to the presence of $f'(u)$. We will explain the details in Section~\ref{sect:boundary-L1radial}. The proof involves introducing the functions $u_\lambda(x):=u(\lambda x)$, interpreting the radial derivative $x\cdot \nabla u (\lambda x)$ as $\frac{d}{d\lambda} u_\lambda (x)$, considering a ``$u_\lambda$-version'' of \eqref{origin:bdry}, and averaging such version  in $\lambda$ to reduce the number of derivatives falling on the solution $u$ through the proof arguments. This will be the crucial and key new idea in the proof: {\it to average the inequalities also with respect to the variable $\lambda$}.

\bigskip

\noindent
{\bf Plan of the paper.} In Section~2 we prove our interior estimate, Theorem \ref{thm:3}. Section~3 is devoted to the boundary bound, Theorem \ref{thm:3bdry}. The paper contains also two Appendices (the first one includes a new interpolation inequality), with results that are needed in Section 3.

\section{The radial derivative controls the function in $L^{1}$}
\label{sect:L1byRad}

In this section we prove, as claimed in Theorem~\ref{thm:3}, that the $L^1$ norm in an annulus (respectively, in a ball) of any superharmonic function  can be controlled, after subtracting its average, by the $L^1$ norm of its radial derivative, taken also in the annulus (respectively, in the ball). For this, we first need a similar result for harmonic functions, with norms taken now on a sphere.

\begin{lemma}\label{lem:RadcontrolsHarm}
Let $v\in C^\infty(\overline B_1)$ solve $\Delta v=0$ in $B_1\subset \R^n$. Then,
\begin{equation}\label{Linfty}
\Vert v-v(0)\Vert_{L^{\infty}(\partial B_{1})}\leq 2n^{3/2}\Vert v_{r}\Vert_{L^{\infty}(\partial B_{1})}
\end{equation} 
and
\begin{equation}\label{L1}
\Vert v-t\Vert_{L^{1}(\partial B_{1})}\leq 2n^{3/2}\Vert v_{r}\Vert_{L^{1}(\partial B_{1})},
\end{equation} 
where $t:=\inf \left\{ \overline{t} : |\{v>\overline{t}\}\cap\partial B_{1}|\leq |\partial B_{1}|/2\right\}$ is the median of $v$.
\end{lemma}

The previous estimates should not be surprising. Indeed, we are controlling the solution to a boundary value problem for an homogeneous equation by its Neumann data ${v_r}_{|_{\partial B_1}}$. Alternatively, the flux ${v_r}_{|_{\partial B_1}}$ can be considered as a first-order elliptic integro-differential operator acting on the function $v_{|_{\partial B_1}}$. This operator is a kind of half-Laplacian, since we extend the data $v_{|_{\partial B_1}}$ harmonically in the ball ---the half-Laplacian corresponds to the harmonic extension in a half-space. Anyhow, these comments will not be needed in the following elementary proof.

\begin{proof}[Proof of Lemma \ref{lem:RadcontrolsHarm}]
We start proving \eqref{Linfty}. From it, we will easily deduce \eqref{L1} by duality. 

Let
$$
w(x):= x\cdot \nabla v (x) = r v_r \qquad\text{ for } x\in \overline B_ 1.
$$
We claim that
\begin{equation}\label{pointw}
|w(x)|\leq 2n^{3/2} \Vert v_{r}\Vert_{L^{\infty}(\partial B_{1})} \, r \qquad\text{ for } x\in \overline B_ 1;
\end{equation}
recall that $r=|x|$. From this, \eqref{Linfty} will clearly follow since, for all $\sigma\in \partial B_1$ we have
$$
v(\sigma)-v(0)=\int_0^1 v_r(r\sigma) \,dr= \int_0^1 r^{-1} w(r\sigma)\,dr.
$$ 

Now, to prove \eqref{pointw}, note that $w$ is harmonic in $B_1$ and agrees with $v_r$ on $\partial B_1$. Consequently,
\begin{equation}\label{Linftyw}
\Vert w\Vert_{L^{\infty}(B_{1})} \leq \Vert v_r\Vert_{L^{\infty}(\partial B_{1})}.
\end{equation}
It follows that \eqref{pointw} holds whenever $|x|=r\geq 1/2$. 

Assume now that $x\in B_{1/2}$. Note that
\begin{equation}\label{wdiff}
|w(x)|= |w(x)-w(0)|\leq \Vert \nabla w \Vert_{L^{\infty}(B_{1/2})} \, r.
\end{equation}
Using that the partial derivatives $w_i$ of $w$ are harmonic and also \eqref{Linftyw}, we see that
\begin{align*}
|w_i(x)|&=\frac{1}{|B_{1/2}|}\left| \int_{B_{1/2}(x)} w_i \, dx \right|= 
\frac{1}{|B_{1/2}|}\left| \int_{\partial B_{1/2}(x)} w \nu^i \, d\HH^{n-1} \right| 
\\
&\leq \frac{|\partial B_{1/2}|}{|B_{1/2}|} \Vert v_r\Vert_{L^{\infty}(\partial B_{1})}
= 2n \Vert v_r\Vert_{L^{\infty}(\partial B_{1})}.
\end{align*}
Hence $|\nabla w(x)|\leq 2n^{3/2} \Vert v_r\Vert_{L^{\infty}(\partial B_{1})}$, which combined with \eqref{wdiff}, gives \eqref{pointw} when  $r< 1/2$.

Next, we establish \eqref{L1} by a duality argument. First note that the value $t$, defined in the statement of Lemma~\ref{lem:RadcontrolsHarm}, is finite and well-defined. Replacing $v$ by $v-t$ we may assume that $t=0$, and therefore that $0=\inf \left\{ \overline{t} : |\{v>\overline{t}\}\cap\partial B_{1}|\leq |\partial B_{1}|/2\right\}$. It follows that $|\{v>0\}\cap\partial B_{1}|\leq |\partial B_{1}|/2$ and that $|\{v>\overline{t}\}\cap\partial B_{1}|> |\partial B_{1}|/2$ for all $\overline{t}<0$. As a consequence, $|\{v\leq\overline{t}\}\cap\partial B_{1}|< |\partial B_{1}|/2$ for all $\overline{t}<0$, and thus $|\{v<0\}\cap\partial B_{1}|\leq |\partial B_{1}|/2$. 

Let now $\text{sgn} (v):=v/|v|$ where $v\neq 0$. Since $|\{v>0\}\cap\partial B_{1}|\leq |\partial B_{1}|/2$ and $|\{v<0\}\cap\partial B_{1}|\leq |\partial B_{1}|/2$, it is possible to extend $ \text{sgn} (v)$ to $\{v=0\}\cap\partial B_{1}$, still taking values $\pm 1$ and in such a way that we have $\int_{\partial B_1} \text{sgn} (v)\, d\HH^{n-1} =0$.
This follows from the fact that, given a measurable set $A\subset \partial B_1$ (which we take here to be $\{v=0\}\cap\partial B_{1}$) and $\theta\in (0,1)$, there exists a measurable subset $B\subset A$ with $|B|=\theta |A|$. This is a consequence of the continuity of the quantity $|A\cap B_\rho(1,0,\ldots,0)|$ with respect to $\rho$. ç

Note that, in addition, we will have $|v|=v\, \text{sgn} (v)$ on $\partial B_1$. 

We now define the functions
$$
g_k (x) := \int_{\partial B_1} \text{sgn} (v)(y)  \, \rho_k (|x-y|)\, d\HH^{n-1} (y),
$$
where $\{\rho_k=\rho_k(|\cdot | )\}$ is a sequence of smooth mollifiers on $\partial B_1$. We have that $g_k\in C^\infty (\partial B_1)$, $|g_k|\leq 1$, and $\int_{\partial B_1} g_k\, d\HH^{n-1} =0$ since $\text{sgn} (v)$ has zero average on~$\partial B_1$. In addition, since $|v|=v\, \text{sgn} (v)$ a.e.\ on $\partial B_1$, it holds that
\begin{equation}\label{limgk}
\int_{\partial B_1} |v|\, d\HH^{n-1} = \lim_k \int_{\partial B_1} v g_k\, d\HH^{n-1}.
\end{equation}

Now, since $g_k$ is smooth and has zero average on $\partial B_1$, we can uniquely solve the problem
$$
\left\{
\begin{array}{cl}
\Delta \varphi_k=0 & \text{in }B_1\\
\partial_r \varphi_k=g_k & \text{on }\partial B_1
\end{array}
\right.
$$
by imposing, additionally, $\varphi_k(0)=0$. By \eqref{Linfty}, we have $\Vert \varphi_k\Vert_{L^{\infty}(\partial B_{1})}\leq 2n^{3/2}$. We conclude that
\begin{align*}
\left| \int_{\partial B_{1}} v g_k \, d\HH^{n-1} \right| & =
\left| \int_{\partial B_{1}} v \, \partial_r \varphi_k \, d\HH^{n-1} \right| =\left| \int_{B_{1}} \nabla v \cdot \nabla\varphi_k \, dx \right|
=\left| \int_{\partial B_{1}} v_r \, \varphi_k \, d\HH^{n-1} \right| 
\\
&
\leq 2n^{3/2}  \int_{\partial B_{1}} |v_r|\, d\HH^{n-1}.
\end{align*}
This, together with \eqref{limgk}, concludes the proof of \eqref{L1}.
\end{proof}

We can now give the

\begin{proof}[Proof of Theorem \ref{thm:3}]
For each of the two inequalities that we need to prove, it suffices to show (with either $E=B_1\setminus B_{1/2}$ or $E=B_1$) that
$$
\Vert u-t_{E}\Vert_{L^{1}(E)}\leq C\Vert u_{r}\Vert_{L^{1}(E)}
$$
for some constant $t_E$ (that may depend on $u$ and $E$). Indeed, once this is shown, it follows that
$$
\left\Vert t_E-\textstyle{\ave_{E}}u  \right\Vert_{L^{1}(E)} = |E|  \left| t_E -  \textstyle{\ave_{E}}u \right| = |E|  \left| \textstyle{\ave_E} (u-t_E) \right| \leq \Vert  u-t_{E} \Vert_{L^{1}({E)}}
$$ 
and hence $\left\Vert u-\textstyle{\ave_{E}}u\right\Vert_{L^{1}(E)}\leq 2C  \Vert u_{r}\Vert_{L^{1}(E)}$
by the triangle inequality.

Now, the proof goes as follows. Since 
$$
\Vert u_{r}\Vert_{L^{1}({B_{1}\setminus B_{1/2})}}= \frac{1}{1/2} \int_{1/2}^1 \left( \int_{\partial B_t} \frac{1}{2} |u_{r}|\, d\HH^{n-1} \right)\, dt,
$$ 
there exists $\rho\in [1/2,1]$ such that $ \Vert u_{r}\Vert_{L^{1}({B_{1}\setminus B_{1/2})}}= \frac{1}{2}\int_{\partial B_\rho}  |u_{r}|\, d\HH^{n-1}$.

Let $v$ be the harmonic function in $B_\rho$ which agrees with $u$ on $\partial B_\rho$. We have $v\in C^\infty (\overline B_\rho)$. Decompose the radial derivatives $u_r=u_r^+-u_r^-$ and $v_r=v_r^+-v_r^-$ in their positive and negative parts. Since $v\leq u$ in $B_\rho$ and the functions agree on the boundary, we have $v_r\geq u_r$ on $\partial B_\rho$. As a consequence, $v_r^-\leq u_r^-$ on $\partial B_\rho$. Since, in addition 
$
0=\int_{B_\rho}  \Delta v\, dx = \int_{\partial B_\rho}  v_{r}\, d\HH^{n-1} = \int_{\partial B_\rho}  v_{r}^+\, d\HH^{n-1}-\int_{\partial B_\rho}  v_{r}^-\, d\HH^{n-1},
$
it follows that
\begin{align*}
2\Vert u_{r}\Vert_{L^{1}({B_{1}\setminus B_{1/2})}} &= \int_{\partial B_\rho}  |u_{r}|\, d\HH^{n-1}
\geq \int_{\partial B_\rho}  u_{r}^-\, d\HH^{n-1} \\
&\geq \int_{\partial B_\rho}  v_{r}^-\, d\HH^{n-1}=\int_{\partial B_\rho}  v_{r}^+\, d\HH^{n-1}= \frac{1}{2}\int_{\partial B_\rho}  |v_{r}|\, d\HH^{n-1}.
\end{align*}

Now, since $\rho\in [1/2,1]$, we can rescale the estimate \eqref{L1} in Lemma~\ref{lem:RadcontrolsHarm} and apply it to the harmonic function $v$ to deduce
$$
C\int_{\partial B_\rho}  |v_{r}|\, d\HH^{n-1} \geq \int_{\partial B_\rho}  |v-t|\, d\HH^{n-1}
=  \int_{\partial B_\rho}  |u-t|\, d\HH^{n-1}
$$
for some value $t$ and a dimensional constant $C$, where in the last equality we used that $v=u$ on $\partial B_\rho$. Therefore, by the previous two estimates, the theorem will be proven once we establish that
\begin{equation}\label{conclsuper}
\Vert u-t\Vert_{L^{1}({B_{1}\setminus B_{1/2})}} \leq C\left( 
\int_{\partial B_\rho}  |u-t|\, d\HH^{n-1}+
\Vert u_r\Vert_{L^{1}({B_{1}\setminus B_{1/2})}}
\right)
\end{equation}
and
\begin{equation}\label{conclsuper-ball}
\Vert u-t\Vert_{L^{1}({B_{1/2})}} \leq C\left( 
\int_{\partial B_\rho}  |u-t|\, d\HH^{n-1}+
\Vert u_r\Vert_{L^{1}({B_{1})}}
\right).
\end{equation}

This is simple. For both estimates, we will use that
$$
(u-t)(s\sigma)=(u-t)(\rho\sigma)-\int_s^\rho u_r (r\sigma) \,dr
$$
for every $s\in (0,1)$ and $\sigma\in S^{n-1}$. 

Now, to check \eqref{conclsuper} we take $s\in (1/2,1)$ and note that
$$
s^{n-1}|(u-t)(s\sigma)|\leq 2^{n-1}\rho^{n-1} |(u-t)(\rho\sigma)|+2^{n-1}\int_{1/2}^1 r^{n-1} |u_r (r\sigma)| \,dr.
$$
Integrating in $\sigma\in S^{n-1}$, and then in $s\in (1/2,1)$, we conclude  \eqref{conclsuper}.

Finally, to prove \eqref{conclsuper-ball} we take $0<s<1/2\leq\rho$ and note that (since $s\leq r\leq\rho$ in the first integral below)
\begin{align*}
s^{n-1}|(u-t)(s\sigma)| &\leq \rho^{n-1} |(u-t)(\rho\sigma)|+\int_{s}^\rho r^{n-1} |u_r (r\sigma)| \,dr \\
& \leq \rho^{n-1} |(u-t)(\rho\sigma)|+\int_{0}^1 r^{n-1} |u_r (r\sigma)| \,dr.
\end{align*}
Integrating in $\sigma\in S^{n-1}$, and then in $s\in (0,1/2)$, we conclude  \eqref{conclsuper-ball} by using also \eqref{conclsuper}.
\end{proof}

\section{The radial derivative controls the function in $L^{1}$ up to the boundary}
\label{sect:boundary-L1radial}

In this section we establish our main result, Theorem \ref{thm:3bdry}. This will be much more delicate than the interior case of Theorem~\ref{thm:3}, which holds assuming only that the function~$u$ is superharmonic. Instead, in the boundary setting, superharmonicity is not enough for the theorem to hold, as the following remark shows. Hence, within the proof we will need to use the semilinear equation satisfied by $u$ and, in fact, also the stability of $u$.

With $\R^n_+=\{x\in\R^n\, :\, x_n>0\}$ and $r=|x|$, throughout this section we use the notation
$$
B_\rho^+=\R^n_+\cap B_\rho, \quad A_{\rho, \overline\rho} := \{\rho< r< \overline\rho\},\quad A^+_{\rho, \overline\rho} := \{x_n>0, \rho< r< \overline\rho\},
$$
and
$$
\partial^0  \Omega= \{x_n=0\}\cap \partial \Omega \quad\text{and}\quad \partial^+  \Omega=\R^n_+\cap \partial \Omega
$$
for an open set $\Omega\subset \R^n_+$.

\begin{remark}\label{rk:not-superh}
For $n\ge 2$, the estimate \eqref{introbdryradial} of Theorem \ref{thm:3bdry}, controlling the function in~$L^1$  in a half-annulus by its radial derivative in $L^1$ (or even in $L^\infty$), also in a half-annulus, cannot not hold within the class of nonnegative superharmonic functions which are smooth in $\overline{\R^n_+}$ and vanish on $\partial \R^n_+$. 

Indeed, for $\delta \in (0,1)$, consider
$$
u^\delta (x):= \frac{x_n}{|(x',x_n+\delta)|} \qquad\text{ for } x=(x',x_n)\in\R^{n-1}\times\R, \, x_n\geq 0.
$$
The function $u^\delta$ is nonnegative and smooth in $\overline{\R^n_+}$, vanishes on $\{x_n=0\}$, and, as a simple computation shows (here one may start from $\Delta u^\delta = x_n \Delta \varphi + 2 \varphi_{x_n}$, where $\varphi=(|x|^2+ 2\delta x_n+\delta^2)^{-1/2}$), is superharmonic in~$\R^n_+$ for $n\geq 2$. At the same time, in the half-annulus it satisfies
$$
|\partial_r u^\delta| = \frac{x_n}{r\, |(x',x_n+\delta)|^3}\, \delta (x_n+\delta) \le \frac{1}{r\,r^{3} }\, \delta (1+\delta)
\le C\delta \quad\text{ in } A^+_{1/2,1}
$$ 
for some constant $C$ independent of $\delta$. By taking $\delta$~small enough, this shows that the estimate  \eqref{introbdryradial} of Theorem \ref{thm:3bdry} cannot hold within this class of functions. 

Note also that in the limiting case $\delta=0$, we are exhibiting a nonnegative  superharmonic function $u^0=x_n/r$ which belongs to $W^{1,2}(B_1^+)$ for $n\geq 3$, vanishes a.e.\ on $\{x_n=0\}$, and is zero homogeneous (i.e., $\partial_r u^0\equiv 0$). 
\end{remark}

When proving Theorem \ref{thm:3bdry}, we will need the following lemma. 
It follows easily from results of \cite{CFRS} which were proved, there, in a quantitative manner.

\begin{lemma}\label{corol:Deltabdry}
Let $u\in C^\infty(\overline{B^+_1})$ be a nonnegative stable solution of $-\Delta u=f(u)$ in~$B^+_1\subset\R^n$, with $u=0$ on $\partial^0 B^+_1$.
Assume that $f\in C^1(\R)$ is nonnegative and nondecreasing. Let $0<\rho_1<\rho_2<\rho_3<\rho_4\leq 1$.

Then,
\begin{equation} \label{prop5.2CFRS-ann}
\|\nabla u\|_{L^{2+\gamma}(A^+_{\rho_2,\rho_3})}
\le C_{\rho_i} \,  \|u\|_{L^{1}(A^+_{\rho_1,\rho_4})}
\end{equation}
and
\begin{equation} \label{hess-ann}
\|D^2 u\|_{L^{1}(A^+_{\rho_2,\rho_3})}
\le C_{\rho_i} \,  \|u\|_{L^{1}(A^+_{\rho_1,\rho_4})}
\end{equation}
for some dimensional constant $\gamma>0$ and some constant $C_{\rho_i}$ depending only on $n$, $\rho_1$, $\rho_2$, $\rho_3$, and $\rho_4$.
\end{lemma}

\begin{proof}
First of all, by \cite[Propositions 2.5 and 5.5]{CFRS} (and a covering of the half-annulus by balls) it suffices to prove the estimates of the lemma with the $L^1$ norm of $u$ on their right-hand sides replaced by the $L^{2}$ norm of $\nabla u$ (in a slightly smaller annulus). The resulting estimates, \eqref{prop5.2CFRS-ann} and \eqref{hess-ann} with $L^1$ norm of $u$ on their right-hand sides replaced by the $L^{2}$ norm of $\nabla u$,
will follow from the following interior and boundary bounds, after making a new covering of the half-annulus by balls.

For \eqref{prop5.2CFRS-ann} we use the $W^{1,2+\gamma}$ estimate of Proposition 2.4 in \cite{CFRS} for the interior balls, and the one of Proposition 5.2 in \cite{CFRS} for the boundary balls ---after a standard covering and scaling argument to change, in these propositions, the radius of the balls in their left and right-hand sides.

Next, to show \eqref{hess-ann}, let $\nu=\nabla u/|\nabla u|$  in the set $\{\nabla u\neq 0\}$ and $\nu=0$ in $\{\nabla u=0\}$. Recall the definition of the quantity $\J$ in (2.7) of \cite{CFRS}. Take an orthonormal basis of $\R^n$  where the last vector is $\nu$. Since $D^2u$ is a symmetric matrix and $\J^2$ contains all the squares of its elements except for those in the last column (which also appear in the last row), we conclude that, up to a multiplicative constant, $\J^2$ controls all the squares $u_{ij}^2$ in the Hessian matrix except for $u_{\nu\nu}^2$. As a consequence, writing $u_{\nu\nu}$ as $\Delta u$ minus the sum of all other diagonal elements, we see that
\begin{equation}
\label{HessbyLapl}
|D^2u|\leq |\Delta u|+C\J \quad\text{ a.e.\ in } B_1^+,
\end{equation}
for some dimensional constant $C$. 

Now, choose a nonnegative function $\zeta\in C^\infty_c(B_{1/2})$ with $\zeta\equiv1$ in $B_{1/4}$.
Then, since $-\Delta u \geq 0$ in $B_1^+$, we have
\begin{equation}\label{int-lapl}
\int_{B_{1/4}}|\Delta u|\,dx
\leq -\int_{B_{1/2}} \Delta u \,\zeta\,dx=\int_{B_{1/2}} \nabla u \cdot \nabla \zeta\,dx \leq C\|\nabla u\|_{L^2(B_{1/2})}.
\end{equation}

To prove \eqref{hess-ann}, we cover the half-annulus $A^+_{\rho_2,\rho_3}$ by balls. For those ones with closure in the open half-space, \eqref{HessbyLapl},  \eqref{int-lapl}, and \cite[Lemma 2.3]{CFRS} lead to the estimate needed to prove \eqref{hess-ann} ---here one uses a scaling argument to replace the balls $B_{1/4}$ and $B_{1/2}$ in \eqref{int-lapl} by balls $B_{s_1}$ and $B_{s_2}$, respectively, for some appropriate $s_1<s_2$).

To complete the proof of \eqref{hess-ann}, we need the analogue Hessian boundary estimate. Since $-\Delta u \geq 0$ in $B_1^+$, with $\zeta$ as above we see that
\begin{equation}\label{massLaplbdry}
\begin{split}
\|\Delta u\|_{L^1({B_{1/4}^+})} & \leq \int_{B_{1/2}^+} (-\Delta u) \,\zeta\, dx=-\int_{\partial^0  B_{1/2}^+} u_\nu \zeta \, d\mathcal H^{n-1}+\int_{B_{1/2}^+} \nabla u\cdot \nabla \zeta\,dx \\ & \leq C \|\nabla u\|_{L^2(B^+_{1})},
\end{split}
\end{equation}
where we have used \cite[Lemma~5.3]{CFRS} (rescaled) to control $u_\nu$. This bound, \eqref{HessbyLapl}, and estimate (5.14) of \cite{CFRS} for $\J$ give the boundary estimate needed to prove \eqref{hess-ann} after a covering and scaling argument.
\end{proof}

Towards the proof of Theorem \ref{thm:3bdry}, to control the $L^1$ norm of a stable solution in a half-annulus by its radial derivative in $L^1$, the starting idea is to use the equation
\begin{equation}\label{continuousbdry}
-2 \Delta u + \Delta (x\cdot\nabla u)  = -f'(u) \, x\cdot\nabla u,
\end{equation}
after multiplying it against an appropriate test function and integrating it in the half-annulus. We will see that this easily yields a lower bound for the integral of the left-hand side which is appropriate for our purposes. The difficulty is how to control the integral of the right-hand side, by above, in terms of only $x\cdot\nabla u= ru_r$. As we will explain within the proof (see the paragraph after \eqref{int-fsxi3}), we wish to use the stability condition to deal with the factor $f'(u)$ in \eqref{continuousbdry}. However, through the simplest approach this would force to control the $L^2$ norm of $\nabla (x\cdot\nabla u)$   ---which is not at our hands since we do not have $L^2$ control on the full Hessian of $u$ (but only the $L^1$ control of Lemma~\ref{corol:Deltabdry}).

To proceed in a similar manner but reducing the number of derivatives falling on~$u$, for $\lambda >0$ we consider the functions 
$$
u_\lambda(x):=u(\lambda x)
$$
and note that $\frac{d}{d\la} u_\la (x)= x \cdot \nabla u (\lambda x)= \lambda^{-1} x \cdot \nabla u_\lambda (x)$.  Using this and noticing that $\lambda^{-2} \Delta u_\lambda = -f(u_\lambda)$, we have
\begin{equation}\label{expression}
\begin{split}
-2\lambda^{-3}\Delta u_{\lambda} + \lambda^{-2} \Delta ( \la^{-1} x\cdot \nabla u_{\lambda}) &
= \frac{d}{d\lambda} \left( \lambda^{-2} \Delta u_\lambda\right) \\
& = - \frac{d}{d\lambda} \, f( u_\lambda )= -f'(u_\lambda)\  \lambda^{-1} x\cdot\nabla u_\lambda.
\end{split}
\end{equation}
Evaluating the first and fourth expressions at $\la=1$, we recover \eqref{continuousbdry}. However, it will be crucial (see the explanation in the paragraph after \eqref{int-fsxi3}) to use the equality between the first and third expressions, after integrating them not only in~$x$, but also in $\lambda$. Integrating in $\lambda$ will be essential to reduce the number of derivatives falling on~$u$. In addition, the monotonicity and convexity of $f$ will allow to use the stability condition appropriately.

This will be understood in all detail going through the following proof.

\begin{proof}[Proof of Theorem \ref{thm:3bdry}]
By rescaling, we may suppose that we have a stable solution in $B_6^+$ instead of $B_1^+$.

We choose a nonnegative smooth function $\zeta$ with compact support in the annulus $A_{4,5}$ and such that $\zeta\equiv 1$ in $A_{4.1,4.9}$. Then, the function $\xi:= x_n \zeta$ satisfies
\begin{equation*}
\begin{split}
 &\xi\ge 0 \text{ in } A_{4,5}^+,  \quad \xi=0 \text{ on }\partial^0 A_{4,5}^+, \\
 & \xi=\xi_{\,\nu} = 0 \text{ on } \partial^+ A_{4,5}^+, \quad \text{and}\quad  \xi=x_n \text{ in } A_{4.1,4.9}^+. 
\end{split}
\end{equation*}

The proof starts from the identity
\begin{equation}\label{derla0-1}
\begin{split}
2\lambda^{-3}\int_{A_{4,5}^+} (-\Delta u_{\lambda}) \, \xi\, dx + \lambda ^{-2}\int_{A_{4,5}^+}  \Delta ( \la^{-1} x\cdot \nabla u_{\lambda}) \, \xi\, dx &\\
&\hspace{-3cm}= - \frac{d}{d\lambda}  \int_{A_{4,5}^+} f(u_\lambda) \, \xi\, dx,
\end{split}
\end{equation}
which follows from the equality between the first and third expressions in \eqref{expression}.
In a first step, we will bound the left-hand side of \eqref{derla0-1} by below. But the subtle part of the proof, where we use the stability of the solution, will be the second step. It will bound the right-hand side of \eqref{derla0-1} by above, but only after averaging it in~$\la$. For this, we will use that
\begin{equation}\label{derla0-2}
\int_1^{1.1} \big( - \frac{d}{d\lambda} \int_{A_{4,5}^+}    f( u_\lambda)\, \xi\, dx \big) d\lambda 
= \int_{A_{4,5}^+} \big( f(u)-f(u_{1.1}) \big) \xi\, dx ,
\end{equation}
where
$$
u_{1.1}:= u_{11/10}=u\big( (11/10) \cdot\big).
$$

\vspace{2mm}\noindent
{\it Step 1: We prove that}
\begin{equation}\label{bdrystep1}
\begin{split}
2\la^{-3}\int_{A_{4,5}^+} (-\Delta u_{\la}) \, \xi\, dx + \la ^{-2}\int_{A_{4,5}^+}  \Delta ( \la^{-1} x\cdot \nabla u_{\la}) \, \xi\, dx &\\
& \hspace{-3cm}\ge c \| u\|_{L^1(A_{4.7,4.8}^+)} - C \|u_r\|_{L^1({A_{3,6}^+})}
\end{split}
\end{equation}
{\it for every $\la\in [1,1.1]$, where $c$ and $C$ are positive dimensional constants.}

To bound by below the second integral in \eqref{bdrystep1} is easy. Since $\xi$ and $\xi_{\,\nu}$ vanish on $\partial^+A^+_{4,5}$, and both $\xi$ and $x\cdot \nabla u_{\la}$ vanish on $\partial^0 A^+_{4,5}$, we deduce that
\begin{equation}\label{2ndint}
\int_{A_{4,5}^+}  \Delta ( x\cdot \nabla u_{\la}) \, \xi\, dx =\int_{A_{4,5}^+}  x\cdot \nabla u_{\la}\, \Delta \xi\, dx \ge -C \|u_r\|_{L^1({A_{4,5.5}^+})}
\end{equation}
since $x\cdot \nabla u_{\la}(x)= \la x\cdot \nabla u(\la x)= \la r \, u_r(\la x)$.

Next, to control the first integral in \eqref{bdrystep1}, given any $\rho_1\in (4.1,4.2)$ and $\rho_2\in (4.8,4.9)$ we consider the solution $\varphi$ of 
\begin{equation}\label{torsion}
\begin{cases} -\Delta \varphi=1 \quad  \quad &\mbox{in } A_{\rho_1,\rho_2}^+
\\
\varphi=0 &\mbox{on  } \partial^0 A_{\rho_1,\rho_2}^+
\\
\varphi_\nu=0 &\mbox{on  } \partial^+ A_{\rho_1,\rho_2}^+.
\end{cases}
\end{equation}
Note that $\varphi\ge 0$ by the maximum principle. In addition, as shown in Appendix~\ref{app:neumann}, we have $|\nabla \varphi|\le C$ in $ A_{\rho_1,\rho_2}^+$ for some dimensional constant $C$. This bound yields $c\varphi\le x_n=\xi$ in $A_{\rho_1,\rho_2}^+$ for some small dimensional constant $c>0$. Since, in addition, $-\Delta u_{\la}$, $\xi$, and $\varphi$ are all nonnegative, it follows that (for positive constants $c$ and $C$ that may differ from line to line)
 \begin{eqnarray*}
\int_{A_{4,5}^+} (-\Delta u_{\la}) \, \xi\, dx &\geq& c\int_{A_{\rho_1,\rho_2}^+} (-\Delta u_{\la})\, \varphi \, dx
\\ & &\hspace{-2cm} = -c \int_{\partial^+ A_{\rho_1,\rho_2}^+} (u_{\la})_\nu \,\varphi \, d\HH^{n-1} 
+c\int_{A_{\rho_1,\rho_2}^+} u_{\la}\, dx
\\ & &\hspace{-2cm} \geq - C \int_{\partial^+ B_{\rho_1}^+} |(u_{\la})_r| \, d\HH^{n-1}  -C \int_{\partial^+ B_{\rho_2}^+} |(u_{\la})_r| \, d\HH^{n-1} 
+c\| u\|_{L^1(A_{\la\rho_1,\la\rho_2}^+)}
\\ & &\hspace{-2cm} \geq - C \int_{\partial^+ B_{\rho_1}^+} |(u_{\la})_r| \, d\HH^{n-1}  -C \int_{\partial^+ B_{\rho_2}^+} |(u_{\la})_r| \, d\HH^{n-1} 
+c\| u\|_{L^1(A_{4.7,4.8}^+)}
\end{eqnarray*}
since $\la\rho_1\le 1.1\cdot 4.2\le 4.7$ and $4.8\le \la\rho_2$. Finally, integrating first in $\rho_1\in (4.1,4.2)$ and then in $\rho_2\in (4.8,4.9)$, we arrive at
\begin{equation*}
 \int_{A_{4,5}^+} (-\Delta u_{\la}) \, \xi\, dx 
\ge -C \|u_r\|_{L^1({A_{4.1,6}^+})}
+c\| u\|_{L^1(A_{4.7,4.8}^+)},
\end{equation*}
since $1.1\cdot 4.9\le 6$.

This and \eqref{2ndint} establish the claim of Step 1.

\vspace{2mm}\noindent
{\it Step 2: We prove that, for every $\ep\in (0,1)$,}
\begin{equation}\label{bdrystep2}
 \int_{A_{4,5}^+} \big( f(u)-f( u_{1.1}) \big) \xi\, dx\le C\big( \varepsilon  \|u\|_{L^{1}(A^+_{3,6})}+ \varepsilon^{-1-2\frac{2+\gamma}{\gamma}} \|u_r\|_{L^1({A_{3,6}^+})}\big)
\end{equation}
{\it for some dimensional constants $\gamma>0$ and $C$ ---with $\gamma$ being the exponent in Lemma~\ref{corol:Deltabdry}. }

By convexity of $f$ we have $f(u)-f( u_{1.1}) \leq f'(u)(u- u_{1.1})$. We now use that $\xi=0$ on $\partial A_{4,5}^+$ and that $u- u_{1.1}=0$ on $\partial^0 A_{3.9,5.1}^+$ in order to take advantage, twice, of the stability of $u$. Taking a function $\phi\in C^\infty_c(A_{3.9,5.1})$ with $\phi=1$ in $A_{4,5}$, and since $\xi$ and $f'(u)$ are nonnegative, we deduce
\begin{align}
 \hspace{0cm} \int_{A_{4,5}^+} \big( f(u)-f( u_{1.1}) \big) \xi\, dx & \leq  \int_{A_{4,5}^+} f'(u)(u-u_{1.1}) \xi\, dx \label{int-fsxi1}
\\ & \hspace{-3cm} 
\leq  \left( \int_{A_{4,5}^+} f'(u)\xi^2\, dx \right)^{1/2}\left( \int_{A_{4,5}^+} f'(u) (u- u_{1.1})^2\, dx \right)^{1/2} \nonumber
\\ & \hspace{-3cm} 
\leq  \left( \int_{A_{4,5}^+} |\nabla \xi|^2\, dx \right)^{1/2} \left( \int_{A_{3.9,5.1}^+} f'(u) \big((u-
 u_{1.1})\phi\big)^2\, dx \right)^{1/2}
\nonumber\\ & \hspace{-3cm} 
\leq  C\left( \int_{A_{3.9,5.1}^+} \left|\nabla \big( (u- u_{1.1})\phi\big)\right|^2\, dx \right)^{1/2}
\nonumber\\ & \hspace{-3cm} 
\leq C\Vert \nabla(u- u_{1.1}) \Vert_{L^2(A_{3.9,5.1}^+)}\label{int-fsxi3},
 \end{align}
where in the last bound we have used Poincar\'e's inequality in $A_{3.9,5.1}^+$ for functions vanishing on $\partial^0 A_{3.9,5.1}^+$. The previous chain of inequalities, which are a crucial part of the proof, have used the stability of the solution twice.

It is worth noticing here that a simple Cauchy-Schwarz argument to control the right-hand side of \eqref{int-fsxi1} would not work, since we do not have control on the integral of $f'(u)^2$. On the other hand, we would be in trouble if we had performed the above chain of inequalities starting from \eqref{continuousbdry} instead of \eqref{expression}. Indeed, in such case, $f'(u)(u- u_{1.1}) \xi $  in \eqref{int-fsxi1}   would be replaced by $f'(u)\,x\cdot\nabla u\,\xi$. Hence, proceeding as we have done above, the final quantity appearing in \eqref{int-fsxi3} would be the $L^2$ norm of $\nabla (x\cdot\nabla u)$. But recall that we do not have control on the $L^2$ norm of the full Hessian of $u$.
 
Next, by the $W^{1,2+\gamma}$ estimate \eqref{prop5.2CFRS-ann} (rescaled to hold in $B_6^+$ instead of $B_1^+$) and taking $q:=\frac{2(1+\gamma)}{2+\gamma}$, we have 
\begin{align} \label{int-fsxibis}
\|\nabla (u- u_{1.1})\|_{L^{2}(A^+_{3.9,5.1})}& \le \|\nabla (u- u_{1.1})\|_{L^{2+\gamma}(A^+_{3.9,5.1})}^{\frac{1}{q}}\|\nabla (u- u_{1.1})\|_{L^{1}(A^+_{3.9,5.1})}^{\frac{1}{q'}}\nonumber\\
&\hspace{-1.5cm} \le C \|\nabla u\|_{L^{2+\gamma}(A^+_{3.9,5.1\cdot 1.1})}^{\frac{1}{q}}\|\nabla(u- u_{1.1})\|_{L^{1}(A^+_{3.9,5.1})}^{\frac{1}{q'}}  \nonumber \\
&\hspace{-1.5cm} \le C   \|u\|_{L^{1}(A^+_{3,6})}^{\frac{1}{q}}  \|\nabla(u- u_{1.1})\|_{L^{1}(A^+_{3.9,5.1})}^{\frac{1}{q'}}\nonumber\\
&\hspace{-1.5cm} \le   \ep \|u\|_{L^{1}(A^+_{3,6})}
+ C\ep^{-\frac{q'}{q}} \|\nabla(u- u_{1.1})\|_{L^{1}(A^+_{3.9,5.1})} 
 \end{align}
for all $\ep\in (0,1)$.
Now, by the interpolation inequality in cubes of Proposition \ref{prop5.2} applied with $p=1$, we claim that
\begin{align}\label{frominterp}
\|{\nabla (u- u_{1.1})}\|_{L^1({A^+_{3.9,5.1}})}&\nonumber\\ 
&\hspace{-2.5cm} \leq C\varepsilon^{1+\frac{q'}{q}} \|{D^2(u- u_{1.1})}\|_{L^1({A^+_{3.8,5.2}})}+ C \varepsilon^{-1-\frac{q'}{q}}\|u- u_{1.1}\|_{L^1({A^+_{3.8,5.2}})}\nonumber\\ 
&\hspace{-2.5cm} \leq C\varepsilon^{1+\frac{q'}{q}} \|{D^2u}\|_{L^1({A^+_{3.8,5.8}})}+ C \varepsilon^{-1-\frac{q'}{q}}\|u- u_{1.1}\|_{L^1({A^+_{3.8,5.2}})}.
\end{align}
To see the first inequality, one covers the half-annulus $A^+_{3.9,5.1}$ (except for a set of measure zero) by cubes of sufficiently small side-length to be contained in $A^+_{3.8,5.2}$. One then applies Proposition \ref{prop5.2} with $p=1$, after rescaling it and renaming $\ep$, in each of these cubes, and finally one adds up all the inequalities.

From \eqref{frominterp} and  the Hessian bound \eqref{hess-ann} (rescaled to hold in $B_6^+$ instead of $B_1^+$) we conclude
\begin{equation*}
\|{\nabla (u- u_{1.1})}\|_{L^1({A^+_{3.9,5.1}})}\leq  \hspace{-.8mm}C\varepsilon^{1+\frac{q'}{q}} \|u\|_{L^{1}(A^+_{3,6})}+ C\varepsilon^{-1-\frac{q'}{q}} \|u- u_{1.1}\|_{L^1({A^+_{3.8,5.2}})}.
\end{equation*}

Putting together this last bound with \eqref{int-fsxi3} and \eqref{int-fsxibis}, and using again \eqref{hess-ann},  we arrive at 
\begin{equation*}
 \int_{A_{4,5}^+} \big( f(u)-f( u_{1.1}) \big) \xi\, dx\le C\varepsilon  \|u\|_{L^{1}(A^+_{3,6})}+ C\varepsilon^{-1-2\frac{2+\gamma}{\gamma}}\|u- u_{1.1}\|_{L^1({A^+_{3.8,5.2}})}.
\end{equation*}
At the same time, since $\frac{d}{d\lambda} u_\lambda (x)=x\cdot\nabla u (\lambda x)= r\, u_r(\lambda x)$, we have
\begin{eqnarray*}
\|u- u_{1.1}\|_{L^1({A_{3.8,5.2}^+})} &=& \int_{A_{3.8,5.2}^+} dx \left| \int_1^{1.1} d\lambda \ r\, u_r(\lambda x)\right| \\
& \le & C\int_1^{1.1}  d\lambda \int_{A_{3.8,5.2\cdot 1.1}^+} dy \ |u_r(y)| \le C\|u_r\|_{L^1({A_{3,6}^+})}.
\end{eqnarray*}
The last two bounds establish Step 2.

\vspace{2mm}\noindent
{\it Step 3: Conclusion.}

We integrate \eqref{derla0-1} in $\la\in (1,1.1)$ and use \eqref{derla0-2}, as well as \eqref{bdrystep1} and \eqref{bdrystep2} in the statements of Steps 1 and 2.
We obtain the bound
\begin{equation}\label{finalbdry}
\| u\|_{L^1(A_{4.7,4.8}^+)} \le C\big( \varepsilon  \|u\|_{L^{1}(A^+_{3,6})}+ \varepsilon^{-1-2\frac{2+\gamma}{\gamma}}\|u_r\|_{L^1({A_{3,6}^+})} \big)
\end{equation}
for all $\ep\in (0,1)$. It is now simple to conclude the desired estimate
\begin{equation}\label{finalB6}
\| u\|_{L^1(A_{3,6}^+)} \le C \|u_r\|_{L^1({A_{3,6}^+})}.
\end{equation}

Indeed, we first show that
\begin{equation}\label{annulusB3}
\| u\|_{L^1(A_{3,6}^+)} \le C  \big( \|u\|_{L^{1}(A^+_{4.7,4.8})}+  \|u_r\|_{L^1({A_{3,6}^+})} \big).
\end{equation}
To prove this, use that 
\begin{equation}\label{alongr}
u(s\sigma)=u(t\sigma)-\int_s^t u_r (r\sigma) \,dr
\end{equation}
for $s\in (3,6)$, $t\in (4.7,4.8)$, and $\sigma\in S^{n-1}$. We deduce
$$
(s/6)^{n-1}|u(s\sigma)|\leq (t/4)^{n-1} |u(t\sigma)|+\int_{3}^6 (r/3)^{n-1} |u_r (r\sigma)| \,dr.
$$
Integrating in $\sigma\in S^{n-1}$, and then in $s\in (3,6)$ and in $t\in (4.7,4.8)$, we conclude  \eqref{annulusB3}.

Now, we use  \eqref{annulusB3} to bound the right-hand side of  \eqref{finalbdry}. In the resulting inequality we choose $\ep\in (0,1)$ small enough so that the constant $C\ep$ multiplying $ \|u\|_{L^{1}(A^+_{4.7,4.8})}$ satisfies $C\ep\le1/2$. We deduce
$$
\| u\|_{L^1(A_{4.7,4.8}^+)} \le C \|u_r\|_{L^1({A_{3,6}^+})},
$$
that, together with \eqref{annulusB3}, yields \eqref{finalB6}. 

Next, to show $\| u\|_{L^1(B_{6}^+)} \le C \|u_r\|_{L^1({B_{6}^+})}$, by \eqref{finalB6} it suffices to show $\| u\|_{L^1(B_{3}^+)} \le C \|u_r\|_{L^1({B_{6}^+})}$. Using again \eqref{finalB6}, it is indeed enough to show 
$$
\| u\|_{L^1(B_{3}^+)} \le C \left(\| u\|_{L^1(A_{3,6}^+)} + \|u_r\|_{L^1({B_{6}^+})}\right).
$$
This is easily shown using \eqref{alongr}, with a similar argument as above or as in the interior case of Theorem~\ref{thm:3}.

The proven bounds hold for a solution in $B_6^+$. By rescaling we conclude the estimates of the theorem for solutions in $B_1^+$.
\end{proof}

For $n\geq 3$, the estimate that we have just proven can be improved as follows.

\begin{remark}\label{rk:bdry-ngeq3}
When $n\geq3$, Theorem \ref{thm:3bdry} also holds when replacing $\| u\|_{L^1(A_{1/2,1}^+)}$, in the left-hand side of~\eqref{introbdryradial}, by $\| u\|_{L^1(B_{1}^+)}$. 

To show this, and given the theorem, it suffices to prove  
$$
\| u\|_{L^1(B_{1/2}^+)} \leq C \| u\|_{L^1(A_{1/2,1}^+)}.
$$ 

But notice that, by the Poincar\' e inequality for functions vanishing on $\partial^0 B^+_{1/2}$, we have that 
\begin{equation}\label{new-ball}
\| u\|_{L^1(B_{1/2}^+)}\le C\|\nabla u\|_{L^{2}(B^+_{1/2})}\le C\|\nabla u\|_{L^{2}(B^+_{4/6})}.
\end{equation}
Now, we use the estimate
\begin{equation}\label{bdry-notweighted}
\int_{B_\rho^+} |\nabla u|^2 \,dx \le C_\lambda \int_{A^+_{\rho,\lambda\rho}} |\nabla u|^2\,dx \qquad \text{if } n\ge 3,
\end{equation}
where $\rho \in (0,1/\lambda)$, $\lambda >1$ is given, and $C_\lambda$ is a constant depending only on $n$ and~$\lambda$. Such estimate follows from Lemma~6.2 in \cite{CFRS}, applied with $\Omega=B_1^+$ and $\theta=0$. Indeed, take $\psi\in C^\infty_c (B_\lambda)$ to be a nonnegative  radially nonincreasing function with $\psi\equiv 1$ in $B_{1}$. For $\rho\in (0,1/\lambda)$, set $\psi_\rho(x) := \psi(x/\rho)$. Note that $|\nabla \psi_\rho|\leq C_\lambda/\rho$ (where $C_\lambda$ depends only on $\lambda$) and that $\nabla \psi_\rho$ vanishes outside of the annulus $A_{\rho,\lambda\rho}$. Choosing $\eta=\psi_\rho$ in the inequality in \cite[Lemma 6.2]{CFRS} with $\theta=0$, for $n\geq 3$ we immediately deduce \eqref{bdry-notweighted}.

Thus, we can bound the last quantity in \eqref{new-ball} computed in a ball by the $L^2$ norm $C\|\nabla u\|_{L^{2}(A^+_{4/6,5/6})}$ computed in an annulus. Finally, by  \eqref{prop5.2CFRS-ann} we know that
$$
\|\nabla u\|_{L^{2}(A^+_{4/6,5/6})}
\le C   \|u\|_{L^{1}(A^+_{1/2,1})},
$$ 
which leads to our claim.
\end{remark}

\medskip\medskip

\appendix

\centerline{\textsc{ \large Appendices}}
%\addtocontents{toc}{\textsc{\hspace{.2cm}  Appendices \vspace{.1cm}}}

\section{Two simple interpolation inequalities in cubes}
\label{app:interp}

In contrast with \cite{CFRS}, the new quantitative proof of the interior result Theorem~\ref{thm:0}, as given in \cite{C}, does not use the $W^{1,2+\gamma}$ bound ---but just the $W^{1,2}$ estimate. This simplification is accomplished thanks to a new interpolation inequality in cubes, Proposition~\ref{prop5.2} below. We prove it here since we have also used it, twice, in our boundary regularity proof, namely, right before and after~\eqref{frominterp}. 

The inequality concerns the quantity $\int |\nabla u|^{p-1}|D^2 u|\, dx$. It is stated in cubes, it does not assume any boundary values, and it is extremely simple to be established. Indeed, in contrast with other interpolation results for functions with no prescribed boundary values, such as \cite[Theorem~7.28]{GT}, proving our inequality in a cube of~$\R^n$ will be immediate once we prove it in dimension one. Instead, the interpolation inequality \cite[Theorem~7.28]{GT} for functions with arbitrary boundary values requires a result on Sobolev seminorms for extension operators (i.e., operators extending functions of~$n$~variables to a larger domain).

We conceived the inequality recently, in our work \cite[Proposition A.3]{CMS}, in order to extend the results of \cite{CFRS} to the $(p+1)$-Laplacian operator. We do not know if the interpolation inequality has appeared somewhere before.

\begin{remark}\label{rem:CMSproof}
The proof that we give here in cubes is exactly the same as that of the last arXiv version of~\cite[arXiv version]{CMS}. Instead, \cite[printed version]{CMS} claimed the same interpolation inequality in balls instead of cubes, and this made the proof (and perhaps the statement) in~\cite[printed version]{CMS} not to be correct. Indeed, proving inequality (A.5) of \cite[printed version]{CMS} contains a mistake: it fails when, given $\varepsilon \in (0,1)$, one has $|I|<<\ep$ ---as it occurs later in the proof in $\R^n$ when some 1d sections of the ball are very small compared to $\ep$. Despite this, all other results of~\cite[printed version]{CMS} remain correct since, within Step 1 of the proof of Theorem~1.1 in \cite[printed version]{CMS},  the interpolation inequality may be applied in small enough cubes covering the ball $B_{1/2}$ (as we do in the current paper) instead of applying it in the whole ball.
\end{remark}

\begin{proposition}\label{prop5.2}
Let $Q=(0,1)^n\subset \R^n$, $p\geq 1$, and $u\in C^{2}(\overline Q)$.

Then, for every $\varepsilon\in (0,1)$,
\begin{equation}\label{5.2}
\int_{Q}\abs{\nabla u}^{p}dx \leq C_p\left( \varepsilon \int_{Q}\abs{\nabla u}^{p-1}\lvert D^2u\rvert\,dx+   \frac{1}{\ep^p} \int_{Q}\abs{u}^{p}dx\right),
\end{equation}
where $C_p$ is a constant depending only on $n$ and $p$.
\end{proposition}

\begin{proof}
We first prove it for $n=1$, in the interval $(0,\delta)$ for a given $\delta\in (0,1)$. Let $u\in C^2([0,\delta])$. Let $x_0\in[0,\delta]$ be such that $\abs{u'(x_0)}= \min_{[0,\delta]}\abs{u'}$.
For $0<y<\frac\delta3<\frac{2\delta}3<z<\delta$, since $(u(z)-u(y))/(z-y)$ is equal to $u'$ at some point, we deduce that $\abs{u'(x_0)}\leq 3\delta^{-1}(|u(y)|+|u(z)|)$. 
Integrating this inequality first in $y$ and later in $z$, we see that
$\abs{u'(x_0)}\leq 9\delta^{-2}\int_{0}^{\delta}\abs{u}\,dx.$
Raising this inequality to the power $p\in [1,\infty)$ we get
\begin{equation}\label{5.2_1}
\abs{u'(x_0)}^{p}\leq 9^{p} \delta^{-p-1}\int_{0}^{\delta}\abs{u}^{p}\,dx.
\end{equation}
Now, for $x\in(0,\delta)$, integrating $\left(\abs{u'}^{p}\right)'$ in the interval with end points $x_0$ and $x$, we deduce
\[
\abs{u'(x)}^{p} \leq p\int_{0}^\delta\abs{u'}^{p-1}\abs{u''}\,dx+\abs{u'(x_0)}^{p}.
\] 
Combining this inequality with \eqref{5.2_1} and integrating in $x\in(0,\delta)$, we obtain
\begin{equation}\label{5.2claim2}
\int_{0}^{\delta}\abs{u'}^{p}\,dx \leq p\,\delta\int_{0}^\delta\abs{u'}^{p-1}\abs{u''}\,dx+9^{p}\delta^{-p}\int_{0}^{\delta}\abs{u}^{p}\,dx.
\end{equation}

We come back now to the statement of the proposition in dimension one and let $u\in C^2([0,1])$. For any given integer $k>1$ we divide $(0,1)$ into $k$ disjoint intervals of length $\delta=1/k$. Since for \eqref{5.2claim2} we did not require any specific boundary values for $u$, we can use the inequality in each of these intervals of length $\delta=1/k$ (instead of in $(0,\delta)$), and then add up all the inequalities, to deduce
\begin{equation}\label{5.2claim3}
\int_{0}^1\abs{u'}^{p}\,dx \leq  \frac{p}{k}\int_{0}^1\abs{u'}^{p-1}\abs{u''}\,dx+(9k)^{p}\int_{0}^1\abs{u}^{p}\,dx.
\end{equation}
Since $0<\ep<1$, there exists an integer $k>1$ such that $\frac{1}{\ep}\leq k<\frac{2}{\ep}$. This and \eqref{5.2claim3} establish the proposition in dimension one.

Finally, for $u\in C^{2}([0,1]^n)$,  denote $x=(x_1,x')\in\R\times\R^{n-1}$. Using \eqref{5.2} with $n=1$ for every $x'$, we get
\begin{equation*}
\begin{split}
\int_{Q}\abs{u_{x_1}}^{p}\,dx 
&=\int_{(0,1)^{n-1}}\,dx'\int_0^1\,dx_1\abs{u_{x_1}(x)}^{p}
\\ 
&\leq C_p\, \varepsilon\int_{(0,1)^{n-1}}\,dx'\int_0^1\,dx_1\abs{u_{x_1}(x)}^{p-1}\abs{u_{x_1x_1}(x)} \\
&\hspace{2cm} +C_p\, \varepsilon^{-p}\int_{(0,1)^{n-1}}\,dx'\int_0^1\,dx_1\abs{u(x)}^{p}
\\
&= C_p\left(\varepsilon\int_{Q}\abs{u_{x_1}(x)}^{p-1}\abs{u_{x_1x_1}(x)}\,dx+ \varepsilon^{-p}\int_{Q}\abs{u(x)}^{p}\,dx\right).
\end{split}
\end{equation*}
Since the same inequality holds for the partial derivatives with respect to each variable $x_k$ instead of $x_1$, adding up the inequalities we conclude \eqref{5.2}. 
\end{proof}

\section{Gradient estimate for a mixed boundary value problem}
\label{app:neumann}

Here we show that the solution to \eqref{torsion}, i.e.,
\begin{equation}\label{torsion-app}
\begin{cases} -\Delta \varphi=1 \quad  \quad &\mbox{in } A_{\rho_1,\rho_2}^+
\\
\varphi=0 &\mbox{on  } \partial^0 A_{\rho_1,\rho_2}^+
\\
\varphi_\nu=0 &\mbox{on  } \partial^+ A_{\rho_1,\rho_2}^+,
\end{cases}
\end{equation}
with $\rho_1\in (4.1,4.2)$ and $\rho_2\in (4.8,4.9)$ ---that we used in Step 1 of the proof of Theorem \ref{thm:3bdry}--- satisfies $|\nabla \varphi|\le C$ in $A_{\rho_1,\rho_2}^+$ for some dimensional constant $C$. Note that the equation is posed in a nonsmooth domain and that this is a mixed Dirichlet-Neumann problem. 

To establish this gradient estimate, we first need to control the $L^\infty$ norm of $\varphi$. For this, recall that, by the maximum principle, $\varphi\ge 0$. Multiplying the equation in~\eqref{torsion-app} by $\varphi^\beta$, with $\beta\ge 1$, and integrating by parts, we get 
\begin{equation}\label{moser}
\int_{A_{\rho_1,\rho_2}^+} |\nabla (\varphi^{(\beta+1)/2})|^2\, dx = \frac{(\beta+1)^2}{4\beta} \int_{A_{\rho_1,\rho_2}^+} \varphi^\beta\, dx.
\end{equation}
We now proceed with a standard iteration procedure, as in the proof of~\cite[Theorem~8.15]{GT}.  To bound by below the left-hand side of \eqref{moser}, we use the Sobolev inequality for functions vanishing on $\partial^0 A_{\rho_1,\rho_2}^+$. Since $\beta +1\geq\beta$, H\"older's inequality yields $\Vert\varphi\Vert_{L^{\chi\cdot\beta}}\leq (C\beta)^{1/\beta}\Vert\varphi\Vert_{L^{\beta}}$ for some dimensional exponent $\chi>1$. Iterating this bound, as in the proof of~\cite[Theorem~8.15]{GT}, we control the $L^\infty$ bound of $\varphi$ by $C\Vert\varphi\Vert_{L^1}$. Finally, notice that
$\Vert\nabla\varphi\Vert_{L^2}^2=\Vert\varphi\Vert_{L^1}$ by \eqref{moser} and that $\Vert\varphi\Vert_{L^1}\leq C \Vert\nabla\varphi\Vert_{L^2}$ by Poincar\'e's inequality for functions vanishing on $\partial^0 A_{\rho_1,\rho_2}^+$. It follows that $\Vert\varphi\Vert_{L^\infty}\leq C\Vert\varphi\Vert_{L^1}\leq C$ for
a dimensional constant $C$.

Now, a simple way to prove the gradient estimate consists of considering the function $\tilde\varphi:=\varphi + x_n^2/2$, which is harmonic and vanishes on the flat boundary. Reflecting it oddly, it becomes a harmonic function in the full annulus $A_{\rho_1,\rho_2}$ (a smooth domain) and has $\tilde\varphi_r= |x_n|x_n/r$ as Neumann data on its boundary. Since this data is a $C^{1,\alpha}$ function on $\partial A_{\rho_1,\rho_2}$, we can now apply global Schauder's regularity theory for the Neumann problem. We use Theorems 6.30 and 6.31 in \cite{GT} to get a $C^{2,\alpha}$ estimate for $\tilde\varphi$ up to the boundary ---and, in particular, a global gradient estimate. The $C^{2,\alpha}$ bound depends only on the $C^{1,\alpha}$ norm of the Neumann data, the $L^\infty$ norm of $\tilde\varphi$ (that we already controlled), and the smoothness of the annulus $A_{\rho_1,\rho_2}$ (which is uniform in $\rho_1$ and $\rho_2$ due to the constraints we placed on these parameters).

\end{document}